\documentclass{article}

\usepackage[utf8]{inputenc} 
\usepackage{amsfonts}
\usepackage{amsmath}
\usepackage{amsthm} 
\usepackage{color}
\usepackage{graphicx}
\usepackage{setspace}
\usepackage{url}
\usepackage{hyperref}
\usepackage{enumerate}
\usepackage{booktabs}
\usepackage{subfigure}

\definecolor{myred}{RGB}{255,50,50}         
\definecolor{myblack}{RGB}{0,0,0}           

\newtheorem{theorem}{Theorem}[section]

\newtheorem{lemma}[theorem]{Lemma}
\newtheorem{definition}[theorem]{Definition}
\newtheorem{corollary}[theorem]{Corollary}
\newtheorem{proposition}[theorem]{Proposition}

\numberwithin{equation}{section}

\newcommand{\grad}{\nabla}                       
\renewcommand{\implies}{\Rightarrow}             
\newcommand{\inner}[2]{\langle#1,#2\rangle}      
\newcommand{\interior}{\mathrm{int}\,}           
\renewcommand{\Re}{\mathbb{R}}                   
\newcommand{\T}{\top\hspace{-1pt}}               
\renewcommand{\S}{\mathcal{S}}                    
\newcommand{\tr}{\mathrm{tr}}                    
\newcommand{\jProd}[2]{ {#1 \circ #2 } }		 
\newcommand{\PSDcone}[1]{{\mathcal{S}^{#1}_+}}	 
\newcommand{\tanCone}[2]{ {\mathcal{T}_{#2}(#1)}}	 
\newcommand{\matRank}{{\mathrm{ rank } \,}}	
\newcommand{\lineality}{\mathrm{lin}\,}

\renewcommand{\doteq}{:=}

\begin{document}


\title{Optimality Conditions for Nonlinear Semidefinite Programming via 
  Squared Slack Variables%
  \thanks{This work was supported by Grant-in-Aid for Young
    Scientists (B) (26730012) and for Scientific Research (C)
    (26330029) from Japan Society for the Promotion of Science.} 
}
\author{
  Bruno F. Louren\c{c}o%
  \thanks{Department of Mathematical and Computing Sciences,
Tokyo Institute of Technology,
2-12-1-W8-41 Ookayama, Meguro-ku, Tokyo 152-8552, Japan 
	(\texttt{flourenco.b.aa@m.titech.ac.jp)}.}
  \and   
  Ellen H. Fukuda%
  \thanks{Department of Applied Mathematics and Physics, 
    Graduate School of Informatics, 
    Kyoto University, Kyoto 606-8501, Japan
    (\texttt{ellen@i.kyoto-u.ac.jp}).}
  \and 
  Masao Fukushima%
  \thanks{
    Department of Systems and Mathematical Science, 
    Faculty of Science and Engineering,
    Nanzan University, Nagoya, Aichi 486-8673, Japan
    \texttt{fuku@nanzan-u.ac.jp}
	}
}

\maketitle

\begin{abstract}
  In this work, we derive second-order optimality conditions for
  nonlinear semidefinite programming (NSDP) problems, by reformulating it
  as an ordinary nonlinear programming problem using squared slack
  variables. We first consider the correspondence between
  Karush-Kuhn-Tucker points and regularity conditions for the general
  NSDP and its reformulation via slack variables. Then, we obtain a
  pair of ``no-gap'' second-order optimality conditions that are
  essentially equivalent to the ones already considered in the
  literature. We conclude with the analysis of some computational
  prospects of the squared slack variables approach for~NSDP.
  
\end{abstract}


\section{Introduction}

We consider the following \emph{nonlinear semidefinite programming} (NSDP)
problem:
\begin{equation}
  \label{eq:sdp}
  \tag{P1}
  \begin{array}{ll}
    \underset{x}{\mbox{minimize}} & f(x) \\ 
    \mbox{subject to} & G(x) \in \PSDcone{m},
  \end{array}
\end{equation}
where $f \colon \Re^n \to \Re$ and $G \colon \Re^n \to \S^m$ are twice
continuously differentiable functions, $\S^m$ is the linear space of
all real symmetric matrices of dimension $m \times m$, and $\S^m_+$ is
the cone of all positive semidefinite matrices in
$\S^m$. \emph{Second-order optimality conditions} for such problems
were originally derived by Shapiro in~\cite{shapiro97}. It might be
fair to say that the second-order analysis of NSDP problems is more
intricated than its counterpart for classical nonlinear programming
problems. That is one of the reasons why it is interesting to have
alternative ways for obtaining optimality conditions
for~\eqref{eq:sdp}; see the works by Forsgren~\cite{FA00} and
Jarre~\cite{Jarre12}. In this work, we propose to use the
\emph{squared slack variables} approach for deriving these optimality
conditions.

It is well-known that the squared slack variables can be used to
transform a \emph{nonlinear programming} (NLP) problem with inequality
constraints into a problem with only equality constraints. For NLP
problems, this technique was hardly considered in the literature
because it increases the dimension of the problem and may lead to
numerical instabilities~\cite{Rob76}. However, recently, Fukuda and
Fukushima~\cite{FF14} showed that the situation may change in the
\emph{nonlinear second-order cone programming} context. Here, we
observe that the slack variables approach can be used also for NSDP
problems, because, like the nonnegative orthant and the
second-order cone, the cone of positive semidefinite matrices is also
a cone of squares. More precisely, $\S^m_+$ can be represented as
\begin{equation}
  \label{eq:sdp_cone}
  \S^m_+ = \{ Z \circ Z \mid Z \in \S^m \},
\end{equation}
where $\circ$ is the \emph{Jordan product} associated with the space
$\S^m$, which is defined as
\begin{displaymath}
  W \circ Z \doteq \frac{WZ + ZW}{2}
\end{displaymath}
for any $W, Z \in \S^m$. Note that actually $Z \circ Z = ZZ = Z^2$ for
any $Z \in \S^m$. 

The fact above allows us to develop the squared slack variables
approach. In fact, by introducing a slack variable $Y \in \S^m$ in
\eqref{eq:sdp}, we obtain the following problem:
\begin{equation}
  \label{eq:sdps}
  \tag{P2}
  \begin{array}{ll}
    \underset{x,Y}{\mbox{minimize}} & f(x) \\
    \mbox{subject to} & G(x) - Y \circ Y = 0,
  \end{array}
\end{equation}
which is an NLP problem with only equality constraints. Note
that if $(x,Y) \in \Re^n \times \S^m$ is a global (local) minimizer
of~\eqref{eq:sdps}, then $x$ is a global (local) minimizer
of~\eqref{eq:sdp}. Moreover, if $x \in \Re^n$ is a global (local)
minimizer of~\eqref{eq:sdp}, then there exists $Y \in \S^m$ such that
$(x,Y)$ is a global (local) minimizer of~\eqref{eq:sdps}.  However,
the relation between stationary points, or \emph{Karush-Kuhn-Tucker}
(KKT) points, of \eqref{eq:sdp} and \eqref{eq:sdps} is not so
trivial. As in~\cite{FF14}, we will take a closer look
at this issue, and investigate also the relation between constraint
qualifications for \eqref{eq:sdp} and~\eqref{eq:sdps}, 
using second-order conditions.

We remark that second-order conditions for these two problems are vastly
different. While \eqref{eq:sdps} is a run-of-the-mill nonlinear
programming problem, \eqref{eq:sdp} has nonlinear conic constraints,
which are more difficult to deal with. Moreover, it is known that
second-order conditions for NSDPs include an extra term, which takes
into account the curvature of the cone. For more details, we refer to
the papers of Kawasaki \cite{Kawasaki88}, Cominetti
\cite{Cominneti1990} and Shapiro \cite{shapiro97}. The main objective
of this work is to show that, under appropriate regularity conditions,
second-order conditions for \eqref{eq:sdp} and \eqref{eq:sdps} are
essentially the same. This suggests that the addition of the slack
term already encapsulates most of the nonlinear structure of the
cone. In the analysis, we also propose and use a sharp
characterization of positive semidefiniteness that takes into account
the rank information. We believe that such a characterization can be
useful in other contexts as well.

Finally, we present results of some numerical experiments where NSDPs are
reformulated as NLPs using slack variables. Note that we are not necessarily
advocating the use of slack variables and we are, in fact, driven by
curiosity about its computational prospects. Nevertheless, there are 
a couple of reasons why this could be interesting. First of all, conventional
wisdom would say that using squared slack variables is not a good
idea, but, in reality, even for linear SDPs there are good reasons to
 use such variables. In \cite{BM03,BM05}, Burer and
Monteiro transform a linear SDP $\inf \{\tr(CX) \mid \mathcal{A}X = b,
X \in \PSDcone{m} \} $  into $\inf \{\tr(CVV^\T) \mid
\mathcal{A}VV^\T = b\}$, where $V$ is a square matrix and $\tr$ denotes
the trace map. The idea is to use a theorem, proven independently by
Barvinok \cite{Barvinok95} and Pataki \cite{Pataki98}, which bounds
the rank of possible optimal solutions. By doing so, it is
possible to restrict $V$ to be a rectangular matrix instead of a
square one, thereby reducing the number of variables. Another reason to
use squared slack variables is that the reformulated NLP problem can
be solved by efficient NLP solvers that are widely available. In
fact, while there are a number of solvers for linear SDPs, as we move
to the general nonlinear case, the situation changes
drastically~\cite{HYY15}.

Throughout the paper, the following notations will be used. For $x \in
\Re^s$ and $Y \in \Re^{\ell \times s}$,  $x_i \in \Re$ and $Y_{ij}
\in \Re$ denote the $i$th entry of $x$ and the $(i,j)$ entry ($i$th row
and $j$th column) of $Y$, respectively. The identity matrix of
dimension $\ell$ is denoted by $I_\ell$. The transpose, the
Moore-Penrose pseudo-inverse, and the rank of $Y \in \Re^{\ell \times s}$
are denoted by $Y^\T \in \Re^{s \times \ell}, Y^{\dagger} \in \Re^{s
  \times \ell}$, and $\matRank Y$, respectively. If $Y$ is a square
matrix, its trace is denoted by $\tr(Y) \doteq \sum_{i} Y_{ii}$. 
For square matrices $W$ and $Y$ of the same dimension, their inner
product\footnote{Although we will also use $\inner{\cdot}{\cdot}$ to denote the inner product in 
$\Re^n$, there will be no confusion.} is denoted by $\inner{W}{Y} \doteq \tr (W^\T Y)$. We will use the
notation $Z \succeq 0$ for $Z \in \PSDcone{m}$. In that case, 
we will denote by $\sqrt{Z}$ the positive semidefinite square root of 
$Z$, that is, $\sqrt{Z}$ satisfies $\sqrt{Z} \succeq 0$ and $\sqrt{Z} ^2 = Z $.
For any function $\mathcal{P} \colon \Re^{s+\ell} \to \Re$, the
gradient and the Hessian at $(x,y) \in \Re^{s+\ell}$ with respect to
$x$ are denoted by $\grad_x \mathcal{P}(x,y)$ and $\grad_x^2
\mathcal{P}(x,y)$, respectively. Moreover, for any linear operator
$\mathcal{G} \colon \Re^s \to \S^\ell$ defined by $\mathcal{G}v =
\sum_{i=1}^s v_i \mathcal{G}_i$ with $\mathcal{G}_i \in \S^\ell$,
$i=1,\dots,s$, and $v \in \Re^s$, the adjoint operator $\mathcal{G}^*:
\S^\ell \to \Re^s$ is defined by
\begin{displaymath}
  \mathcal{G}^* Z = (\inner{\mathcal{G}_1}{Z}, \dots,
  \inner{\mathcal{G}_s}{Z})^\T, \quad Z \in \S^\ell.
\end{displaymath}
Given a mapping $\mathcal{H} \colon \Re^s \to \S^\ell$, its derivative
at a point $x \in \Re^s$ is denoted by $\grad \mathcal{H}(x) \colon
\Re^s \to \S^\ell$ and defined by
\begin{displaymath}
  \grad \mathcal{H}(x) v = \sum_{i=1}^s v_i \frac{\partial
    \mathcal{H}(x)}{\partial x_i}, \quad v \in \Re^s,
\end{displaymath}
where $\partial \mathcal{H}(x)/\partial x_i \in \S^\ell$ are the partial
derivative matrices. Finally, for a closed convex cone $K$, we will
denote by $\lineality K$ the largest subspace contained in $K$. Note
that $\lineality K = K\cap -K$.

The paper is organized as follows. In Section~\ref{sec:prel}, we
recall a few basic definitions concerning KKT points and second-order
conditions for \eqref{eq:sdp} and \eqref{eq:sdps}. We also give
a sharp characterization of positive semidefiniteness. In
Section~\ref{sec:kkt}, we prove that the original and the reformulated
problems are equivalent in terms of KKT points, under some conditions.
In Section~\ref{sec:cq}, we establish the relation between constraint
qualifications of those two problems. The analyses of second-order
sufficient conditions and second-order necessary conditions are
presented in Sections~\ref{sec:sosc} and~\ref{sec:sonc}, respectively.
In Section~\ref{sec:comp}, we show some  computational
results. We conclude in Section~\ref{sec:conclusion}, with 
final remarks and future works.


\section{Preliminaries}\label{sec:prel}
\subsection{A sharp characterization of positive semidefiniteness}

It is a well-known fact that a matrix $\Lambda  \in \S^m$ is positive semidefinite if and only 
if $\inner{\Lambda }{W^2} \geq 0$ for all $W \in \S^m$. This statement is equivalent 
to the self-duality of the cone $S^m_+$. However, we get no information about the rank of 
$\Lambda $. In the next lemma, we give a new characterization of positive semidefinite 
matrices, which takes into account the rank information.

\begin{lemma}\label{lemma:psd}
Let $\Lambda \in \S^m$. The following statements 
are equivalent:
\begin{enumerate}[(i)]
  \item $\Lambda \in \PSDcone{m}$;
  \item There exists $Y \in \S^m$ such that $\jProd{Y}{\Lambda} = 0$ and $Y \in \Phi(\Lambda)$,
    where
    \begin{equation}\label{eq:lemma:psd}
      \Phi(\Lambda) := \big\{ Y  \in \S^m \,|\, \inner{\jProd{W}{W}}{\Lambda} > 0 \mbox{ for all } 0 \ne W \in \S^m \mbox{ with } \jProd{Y}{W} = 0 \big\}.
    \end{equation} 
\end{enumerate}
For any $Y$ satisfying the conditions in $(ii)$, we have  $\matRank
\Lambda = m - \matRank Y$. Moreover, if $\sigma$ and $\sigma '$ are
nonzero eigenvalues of $Y$, then $\sigma + \sigma ' \neq 0$.
\end{lemma}

\begin{proof}
Let us prove first that $(ii)$ implies $(i)$. Since the inner product is invariant 
under orthogonal transformations,  we may assume without loss of generality
that $Y$ is diagonal, i.e.,
\begin{equation*}
Y = \begin{pmatrix} D & 0 \\ 0 & 0\end{pmatrix},
\end{equation*}
where $D$ is a $k \times k$ nonsingular diagonal matrix, and $\matRank Y = k$. 
We partition $\Lambda$ in blocks in a similar way:
\begin{equation*}
\Lambda = \begin{pmatrix}
A &B\\ B^\T&C
\end{pmatrix},
\end{equation*}
where $A \in \S^k, B\in \Re^{k\times (m-k)}$, and $C \in \S^{m-k}$. We will proceed by proving that
$A = 0$, $B = 0$ and $C$ is positive definite. 

First, observe that, by assumption,
\begin{equation}
  \label{eq:lemma_psd.1}
 0 = \jProd{ Y}{\Lambda} = \begin{pmatrix}  \jProd{D}{A} & DB/2 \\ B^\T D/2 & 0 \end{pmatrix} 
\end{equation}
holds. Since $D$ is nonsingular, this implies $B = 0$. Now, let us prove that $A = 0$.
From~\eqref{eq:lemma_psd.1} and the fact that $D$ is diagonal, we obtain  
\begin{equation}\label{eq:lemma}
0 = 2 (\jProd{D}{A})_{ij} = A_{ij}(D_{ii} + D_{jj}).
\end{equation}
Again, because $D$ is nonsingular, it must be the case that all diagonal elements of $A$ should 
be zero. Now, suppose that $A_{ij}$ is nonzero for some $i$ and $j$, with $i \neq j$. In face of \eqref{eq:lemma}, 
this can  only happen if $D_{ii} + D_{jj} = 0$. Let us now consider the following matrix:
\begin{equation*}
  W = \begin{pmatrix} \widetilde W & 0 \\ 0 & 0 \end{pmatrix} \in \S^m, 
\end{equation*}
where $\widetilde W \in \S^k$ is a submatrix containing only two nonzero elements, $\widetilde W_{ij} = 1$ and 
$\widetilde W_{ji} = 1$. Then, easy calculations show that $\jProd{\widetilde W}{D} = 0$, which also implies $\jProd { W}{Y} = 0$. 
Moreover, $\inner{W \circ W}{ \Lambda } = 0$ because $\widetilde W^2 = \jProd{\widetilde W}{\widetilde W}$ is the diagonal matrix 
having $1$ in the $(i,i)$ entry and $1$ in the $(j,j)$ entry, and since $A_{ii}$ and $A_{jj}$ are zero. 
We conclude that $Y \notin \Phi(\Lambda)$, contradicting the assumptions. So, it follows that $A$ must be zero. 
Similarly, we have that $D_{ii} + D_{jj}$ is never zero, which corresponds 
to the statement about eigenvalues $\sigma$ and $\sigma '$ in the lemma. In fact,
if $D_{ii} + D_{jj}$ is zero, then, by taking $W$ exactly as before, we have
$\jProd {W}{Y} = 0$ and  $\inner{W \circ W}{ \Lambda } = 0$. Once again, this shows that $Y \notin \Phi(\Lambda)$,
which is a contradiction.
 
It remains to show that $C$ is positive definite. Taking an arbitrary nonzero $\tilde{H} \in \S^{m-k}$, and defining
\begin{equation*}
  H = \begin{pmatrix} 0 &0\\ 0&\tilde{H} \end{pmatrix} \in \S^m, 
\end{equation*}
we easily obtain $H \circ Y = 0$. Since $Y \in \Phi(\Lambda)$, we have $\inner{H^2}{\Lambda} > 0$.
But this shows that $\inner{\tilde{H}^2}{C} > 0$, which implies that $C$ is positive definite. In particular, 
the rank of $\Lambda$ is equal to the rank of $C$, which is $m - \mathrm{rank } \,Y$.

Now, let us prove that $(i)$ implies $(ii)$. Similarly, we may
assume $\Lambda = \bigl(\begin{smallmatrix} 0 &0\\ 0&C \end{smallmatrix} \bigr)$, with 
$C$ positive definite. Then, we can take
$Y = \bigl(\begin{smallmatrix} E &0\\ 0&0\end{smallmatrix} \bigr)$, 
where $E$ is any positive definite matrix. It follows that any matrix 
$W \in \S^m$ satisfying  $\jProd{Y}{W}  = 0$ 
must have the shape $\bigl(\begin{smallmatrix} 0 &0\\ 0&F \end{smallmatrix} \bigr)$, for 
some matrix $F$.
Since $C$ is positive definite, it is clear that $\inner{\jProd{W}{W}}{ \Lambda } > 0 $, whenever
$W$ is nonzero. 
\end{proof}

The statement about the sum of nonzero eigenvalues might seem innocuous at 
first, but it will be very useful in Section \ref{sec:sosc}. In fact, the idea for this new 
characterization of positive semidefiniteness comes from the second-order conditions of~\eqref{eq:sdps}. 
For now, let us present another result that will be necessary. Given $A \in \S^m$, denote by 
$L_A \colon \S^m \to \S^m$ the linear operator defined by
\begin{equation*}
L_A(E) := \jProd{A}{E}
\end{equation*}
for all $E \in \S^m$. There are many examples of invertible matrices 
$A$ for which the operator $L_A$ is not invertible\footnote{Take $A = \bigl(\begin{smallmatrix} 1 &0\\ 0&-1 \end{smallmatrix} \bigr)$, for example.}. 
This is essentially due to the failure of the condition 
on the eigenvalues. The following proposition is well-known in the context 
of Euclidean Jordan algebra (see \cite[Proposition~1]{Sturm2000}), 
but we include here a short-proof for completeness.

\begin{proposition}\label{prop:nonsigular}
Let $A \in \S^m$. Then, $L_A$ is invertible if and only if $\sigma + \sigma ' \neq 0$ for 
every pair of eigenvalues $\sigma, \sigma '$ of $A$; in this case, $A$ must be invertible.
\end{proposition}
\begin{proof}
The statements in the proposition are all invariant under orthogonal transformations.
Thus,  we may assume without loss of generality that $A$ is already diagonalized, 
and so $A_{kk}$ is an eigenvalue of $A$ for every $k = 1,\ldots, m$. 

Let us show that the invertibility of $L_A$ implies the statement about the eigenvalues of $A$. We will 
do so by proving the contrapositive.
Take $i$ and $j$ such that $A_{ii} + A_{jj} = 0$. 
Let $W$ be such that all the entries are zero except for $W_{ij} = W_{ji} = 1$. 
Then, we have $\jProd{A}{W} = 0$. This shows that the kernel of $L_A$ is non-trivial and consequently, 
$L_A$ is not invertible.

Reciprocally, since we assume that $A$ is diagonal, for every $W \in \S^m$, 
we have $2(L_A(W))_{ij} = W_{ij}(A_{ii} + A_{jj})$ for all $i$ and $j$. Due to the fact 
that $A_{ii} + A_{jj}$ is never zero, the kernel of $L_A$ must only contain 
the zero matrix. Hence $L_A$ is invertible, and the result follows.
\end{proof} 
In view of Proposition \ref{prop:nonsigular}, the matrix $D$ which appears 
in the proof of Lemma \ref{lemma:psd} is such that $L_D$ is invertible.  
This will play an important role when we discuss the relation between the
second-order sufficient conditions of problems \eqref{eq:sdp} and \eqref{eq:sdps}.

\subsection{KKT conditions and constraint qualifications}

Now, let us consider the following lemma, which will allow us to
present appropriately the KKT conditions of problems~\eqref{eq:sdp}
and \eqref{eq:sdps}.

\begin{lemma}
  \label{lem:matrices}
  The following statements hold.
  \begin{itemize}
  \item[(a)] For any matrices $A,B \in \Re^{m \times m}$, let $\varphi
    \colon \Re^{m \times m} \to \Re$ be defined by $\varphi(Z) \doteq
    \tr(Z^\T A Z B)$. Then, we have $\grad \varphi(Z) = AZB + A^\T Z
    B^\T$.
  \item[(b)] For any matrix $A \in \S^m$, let $\varphi \colon \S^m
    \to \Re$ be defined by $\varphi(Z) \doteq \inner{Z \circ
      Z}{A}$. Then, we have $\grad \varphi(Z) = 2\jProd {Z}{A} $.
  \item[(c)] For any matrix $A \in \Re^{m \times m}$ and function
    $\theta \colon \Re^n \to \S^m$, let $\psi \colon \Re^n \to \Re$
    be defined by $\psi(x) = \inner{\theta(x)}{A}$. Then, we have
    $\grad \psi(x) = \grad \theta(x)^* A$.
  \item[(d)] Let $A,B \in \S^m$. Then, they commute, i.e., $AB = BA$,
    if and only if $A$ and $B$ are simultaneously diagonalizable by an
    orthogonal matrix, i.e., there exists an orthogonal matrix $Q$
    such that $QAQ^\T$ and $QBQ^\T$ are diagonal.
  \item[(e)] Let $A, B \in \S^m_+$. Then, $AB = 0$ if and only if
    $\inner{A}{B} = 0$.
  \end{itemize}
\end{lemma}

\begin{proof}
  (a) See \cite[Section 10.7]{Ber09}.\\ 

  \noindent (b) Note that $\varphi(Z) =  \inner{\jProd{Z}{Z}}{\Lambda} = 
  \frac{1}{2}\inner{ZZ^\T}{A} + \frac{1}{2}\inner{Z^\T Z}{A}  = \frac{1}{2}\tr(ZZ^\T A) +  \frac{1}{2}\tr(Z^\T ZA) = 
  \frac{1}{2}\tr(Z^\T AZ) +  \frac{1}{2}\tr(Z^\T ZA) $.  Let $\varphi _1(Z) = \tr(Z^\T AZ)$ and 
  $\varphi _2(Z) = \tr(Z^\T Z A)$. Then, from item (a), we have $\grad \varphi _1(Z) = AZ + A^\T Z$
  and $\grad \varphi _2(Z) = ZA + ZA^\T$. Taking into account the symmetry of $A$, 
  we have $\grad \varphi _1(Z) = 2AZ$ and $\grad \varphi _2(Z) = 2ZA$. Hence we have 
  $\grad \varphi (Z) = \frac{1}{2} \grad \varphi _1(Z) + \frac{1}{2} \grad \varphi _2(Z) = 2\jProd{A}{Z}$. \\

  \noindent (c) Observe that $\psi(x) = \inner{\theta(x)}{A} =
  \tr(\theta(x)A) = \sum_{i,j} \theta(x)_{ij} A_{ij}$ for any $x \in
  \Re^n$. Then, we have
  \begin{displaymath}
    \grad \psi(x) = 
    \left(
    \begin{array}{c}
      \sum_{i,j} (\partial \theta(x)_{ij}/\partial x_1) A_{ij} \\
      \vdots \\
      \sum_{i,j} (\partial \theta(x)_{ij}/\partial x_n) A_{ij}
    \end{array}
    \right)
    = 
    \left(
    \begin{array}{c}
      \inner{\partial \theta(x)/\partial x_1}{A}\\
      \vdots \\
      \inner{\partial \theta(x)/\partial x_n}{A}\\
    \end{array}
    \right)
    = \grad \theta(x)^* A,
  \end{displaymath}
  where the last equality follows from the definition of adjoint
  operator.\\

  \noindent (d) See  \cite[Section 8.17]{Ber09}.\\

  \noindent (e) See  \cite[Section 8.12]{Ber09}.
\end{proof}

We can now recall the KKT conditions of problems~\eqref{eq:sdp}
and~\eqref{eq:sdps}. First, define the Lagrangian function $L \colon
\Re^n \times \S^m \to \Re$ associated with problem~\eqref{eq:sdp} as
\begin{displaymath}
  L(x,\Lambda) \doteq f(x) - \inner{G(x)}{\Lambda}.
\end{displaymath}
We say that $(x,\Lambda) \in \Re^n \times \S^m$ is a KKT pair of
problem~\eqref{eq:sdp} if the following conditions are satisfied:
\begin{align}
  \grad f(x) - \grad G(x)^* \Lambda =  0, \label{eq:kkt_sdp.1} \tag{P1.1}\\ 
  \Lambda \succeq 0, \label{eq:kkt_sdp.2} \tag{P1.2}\\ 
  G(x) \succeq 0, \label{eq:kkt_sdp.3} \tag{P1.3}\\ 
  \jProd{\Lambda}{G(x)}  = 0, \label{eq:kkt_sdp.4} \tag{P1.4}
\end{align}
where, from Lemma~\ref{lem:matrices}(c), we have 
$\grad f(x) - \grad G(x)^* \Lambda = \grad_x L(x,\Lambda)$. Applying the trace map on both
sides of \eqref{eq:kkt_sdp.4}, we see that
condition~\eqref{eq:kkt_sdp.4} is equivalent to $\inner{\Lambda}{G(x)}
= 0$. This result, together with the fact that $\Lambda \succeq 0$ and
$G(x) \succeq 0$, shows that \eqref{eq:kkt_sdp.4} is also equivalent
to $\Lambda G(x) = 0$, by Lemma~\ref{lem:matrices}(e). Moreover, the
equality \eqref{eq:kkt_sdp.4} implies that $\Lambda$ and $G(x)$
commute, which means, by Lemma~\ref{lem:matrices}(d), that they are
simultaneously diagonalizable by an orthogonal matrix. The following
definition is also well-known.

\begin{definition}\label{def:strict}
  If $(x,\Lambda) \in \Re^n \times \S^m$ is a KKT pair of
  \eqref{eq:sdp} such that
  \begin{displaymath}
    \matRank G(x) + \matRank \Lambda = m,
  \end{displaymath}
  then $(x,\Lambda)$ is said to satisfy the \emph{strict
    complementarity} condition.
\end{definition}

As for the equality constrained NLP problem~\eqref{eq:sdps}, we observe
that $(x,Y,\Lambda) \in \Re^n \times \S^m \times \S^m$ is a KKT triple
if the conditions below are satisfied:
\begin{align*}
  \grad_{(x,Y)} \mathcal{L}(x,Y,\Lambda) & = 0, \\
  G(x) - Y \circ Y & = 0,
\end{align*}
where $\mathcal{L} \colon \Re^n \times \S^m \times \S^m \to \Re$ is
the Lagrangian function associated with~\eqref{eq:sdps}, which is given by
\begin{displaymath}
  \mathcal{L}(x,Y,\Lambda) \doteq f(x) - \inner{G(x) - Y \circ
    Y}{\Lambda}.
\end{displaymath}
From Lemma~\ref{lem:matrices}(b),(c), these conditions can be written as
follows:
\begin{align}
  \grad f(x) - \grad G(x)^* \Lambda = 0, \label{eq:kkt_sdps.1} 
  \tag{P2.1}\\
  \jProd{\Lambda}{Y} = 0, \label{eq:kkt_sdps.2} \tag{P2.2}\\
  G(x) - Y \circ Y = 0. \label{eq:kkt_sdps.3} \tag{P2.3}
\end{align}

For problem \eqref{eq:sdp}, we say that the \emph{Mangasarian-Fromovitz constraint 
qualification} (MFCQ) holds at a point $x$ if there exists some $d \in \Re^n$ such that 
\begin{equation*}
G(x) + \grad G(x)d \in \interior \PSDcone{m},
\end{equation*}
where $\interior \PSDcone{m} $ denotes the interior of $\PSDcone{m} $, that 
is, the set of symmetric positive definite matrices.
If $x$ is a local minimum for \eqref{eq:sdp}, MFCQ ensures the existence of 
a Lagrange multiplier $\Lambda$ and that the set of multipliers is bounded.
A more restrictive assumption is the nondegeneracy condition discussed 
in \cite{shapiro97}, where it is presented in terms of a transversality condition on 
the map $G$. However,
at the end, it boils down to the following condition.
\begin{definition}\label{def:nondeg}
  Suppose that $(x,\Lambda) \in \Re^n \times \S^m$ is a KKT pair of
  \eqref{eq:sdp} such that
\begin{displaymath}
 \S^m = \lineality \tanCone{G(x)}{\PSDcone{m}} + \mathrm{Im}\,\grad G(x),
\end{displaymath}
where $\mathrm{Im}\,\grad G(x)$ denotes the image of the linear map $\grad G(x)$, $\tanCone{G(x)}{\PSDcone{m}}$ denotes
  the tangent cone of $\S^m_+$ at $G(x)$, and  $\lineality \tanCone{G(x)}{\PSDcone{m}}$ is the lineality space of 
the tangent cone $\tanCone{G(x)}{\PSDcone{m}}$, i.e., $\lineality \tanCone{G(x)}{\PSDcone{m}} = \tanCone{G(x)}{\PSDcone{m}} \cap -\tanCone{G(x)}{\PSDcone{m}}$ 
(see, for instance, the observations on page 310 in \cite{shapiro97}).
  Then, $(x,\Lambda)$ is said to satisfy the \emph{nondegeneracy} condition.
\end{definition}
A good thing about the nondegeneracy condition is that it ensures 
that $\Lambda$ is unique.

For problem \eqref{eq:sdps}, a common constraint qualification is 
the \emph{linear independence constraint qualification} (LICQ), which simply 
requires that the gradients of the constraints be linearly independent. In 
Section \ref{sec:cq}, we will show that LICQ and the nondegeneracy are 
essentially equivalent.

\subsection{Second-order conditions}

Since \eqref{eq:sdps} is just an ordinary equality constrained nonlinear program, 
second-order sufficient conditions are well-known and can be written as 
follows. 

\begin{proposition}
  Let $(x,Y,\Lambda) \in \Re^n \times \S^m \times \S^m$ be a KKT triple of problem~\eqref{eq:sdps}. The
  \emph{second-order sufficient condition} (SOSC-NLP)\footnote{We refer to this condition as SOSC-NLP in
    order to distinguish it from SOSC for~SDP.} holds if
  \begin{equation}
    \inner{\grad _x ^2 L(x,\Lambda)v}{v} + 2\inner{\jProd{W}{W}}{\Lambda} > 0 \label{eq:sosc_nlp.1}
  \end{equation}
  for every nonzero $(v,W) \in \Re^n \times \S^m$ such that $\grad G(x)v - 2\jProd{Y}{W} = 0$.
\end{proposition}
\begin{proof}
The second-order sufficient condition for~\eqref{eq:sdps} holds if 
  \begin{displaymath}
    \inner{\grad _{(x,Y)} ^2 \mathcal{L}(x,Y,\Lambda)(v,W)}{(v,W)} > 0
  \end{displaymath}
for every nonzero $(v,W) \in \Re^n \times \S^m$ such that $\grad G(x)v - 2\jProd{Y}{W} = 0$; see 
\cite[Section 3.3]{Ber99} or \cite[Theorem 12.6]{NW99}.
Since
\begin{equation*}
\inner{\grad _{(x,Y)} ^2 \mathcal{L}(x,Y,\Lambda)(v,W)}{(v,W)} = \inner{\grad _x ^2 L(x,\Lambda)v}{v} + 2\inner{\jProd{W}{W}}{\Lambda},
\end{equation*}
we have the desired result.
\end{proof}
Similarly, we have the following second-order  necessary condition. Note that we require the 
LICQ to hold.

\begin{proposition}
Let $(x,Y)$ be a local minimum for \eqref{eq:sdps} 
and $(x,Y,\Lambda) \in \Re^n \times \S^m \times \S^m$ be a KKT triple such 
that LICQ holds. Then, the following second-order necessary condition (SONC-NLP) holds: 
\begin{equation}
    \inner{\grad _x ^2 L(x,\Lambda)v}{v} + 2\inner{\jProd{W}{W}}{\Lambda} \geq 0 \label{eq:sonc_nlp}
  \end{equation}
  for every $(v,W) \in \Re^n \times \S^m$ such that $\grad G(x)v - 2\jProd{Y}{W} = 0$.
\end{proposition}
\begin{proof}
  See \cite[Theorem 12.5]{NW99}.
\end{proof}

Second-order conditions for \eqref{eq:sdp} are a more delicate matter. Let $(x,\Lambda )$ be a KKT 
pair of \eqref{eq:sdp}. It is true that a \emph{sufficient} condition for optimality 
is that the Hessian of the Lagrangian be positive definite over the 
set of critical directions. However, replacing ``positive definite'' by  
``positive semidefinite'' does not yield a necessary condition. Therefore, 
it seems that there is a gap between necessary and sufficient conditions. In order 
to close the gap, it is essential to add an additional term to the Hessian of the Lagrangian. For the 
theory behind this see, for instance, the papers by Kawasaki~\cite{Kawasaki88}, Cominetti~\cite{Cominneti1990}, and 
Bonnans, Cominetti and Shapiro~\cite{BonnansCS}.
The condition below was obtained by Shapiro in~\cite{shapiro97} and it is sufficient for 
$(x,\Lambda)$ to be a local minimum, see Theorem 9 therein. 




\begin{proposition}
  Let $(x,\Lambda) \in \Re^n \times \S^m$ be a KKT pair of problem~\eqref{eq:sdp} 
  satisfying strict complementarity and the nondegeneracy condition. The
  second-order sufficient condition (SOSC-SDP) holds if
  \begin{equation}\label{eq:sosc_sdp.1}
    \inner{(\grad_x^2L(x,\Lambda) + H(x,\Lambda))d}{d} > 0
  \end{equation}
  for all nonzero $d \in \mathcal{C}(x)$, where 
  \begin{displaymath}
    \mathcal{C}(x) \doteq \left\{ d \in \Re^n \, \middle| \, \grad G(x)d \in
    \tanCone{G(x)}{\PSDcone{m}}, \inner{\grad f(x)}{ d }= 0
    \right\}
  \end{displaymath}
  is the critical cone at $x$,  and $H(x,\Lambda)\in \S^m$ is a matrix 
  with elements
  \begin{equation}
    H(x,\Lambda)_{ij} \doteq 2 \, \tr \left( \frac{\partial
      G(x)}{\partial x_i} G(x)^{\dagger} \frac{\partial G(x)}{\partial
      x_j} \Lambda \right)\label{eq:sosc_sdp.2}
  \end{equation}
  for $i,j = 1, \ldots , n$. In this case, $(x,\Lambda)$ is a local minimum for 
  \eqref{eq:sdp}. Conversely, if $x$ is a local minimum for \eqref{eq:sdp} 
  and $(x,\Lambda)$ is a KKT pair satisfying strict complementarity and
  nondegeneracy, then the following 
  second-order necessary condition (SONC-SDP) holds:
  \begin{equation}\label{eq:sonc_sdp}
    \inner{(\grad_x^2L(x,\Lambda) + H(x,\Lambda))d}{d} \geq 0
  \end{equation}
  for all $d \in \mathcal{C}(x)$.  
\end{proposition}


\section{Equivalence between KKT points}\label{sec:kkt}

Let us now establish the relation between KKT points of the original
problem~\eqref{eq:sdp} and its reformulation~\eqref{eq:sdps}. We start
with the following simple implication.

\begin{proposition}
  \label{prop:sdp_to_sdps}
  Let $(x,\Lambda) \in \Re^n \times \S^m$ be a KKT pair of problem
  \eqref{eq:sdp}. Then, there exists $Y \in \S^m$ such that
  $(x,Y,\Lambda)$ is a KKT triple of \eqref{eq:sdps}.
\end{proposition}

\begin{proof}
 Let $Y \in \PSDcone{m}$ be the positive semidefinite matrix satisfying $G(x) = \jProd{Y}{Y}$. Let 
 us show that $(x,Y,\Lambda)$ is a KKT triple of \eqref{eq:sdps}.
  The conditions~\eqref{eq:kkt_sdps.1} and \eqref{eq:kkt_sdps.3} are immediate. 
  We need to show that \eqref{eq:kkt_sdps.2} holds. 
  
  Recall that~\eqref{eq:kkt_sdp.4} along with~\eqref{eq:kkt_sdp.2} 
  and~\eqref{eq:kkt_sdp.3}
  implies $G(x)\Lambda = 0$, due to Lemma~\ref{lem:matrices}(e). It means that every column of $\Lambda$ lies 
  in the kernel of $G(x)$. However, $G(x)$ and $Y$ share exactly the same 
  kernel, since $G(x) = Y^2$. It follows that $Y\Lambda = 0$, so that $\jProd{Y}{\Lambda} = 0$.  

\end{proof}

The converse is not always true. That is, even if $(x,Y,\Lambda)$ is a KKT triple
of \eqref{eq:sdps}, $(x,\Lambda)$ may fail to be a KKT pair 
of \eqref{eq:sdp}, since $\Lambda$ need not be positive semidefinite. 
This, however, is the only obstacle for establishing equivalence. 

\begin{proposition}
If $(x,Y,\Lambda) \in \Re^n \times \S^m \times \S^m$ is a KKT triple of \eqref{eq:sdps} such that  $\Lambda$ is positive semidefinite, 
then $(x,\Lambda)$ is a KKT pair of \eqref{eq:sdp}.
\end{proposition}
\begin{proof}
The only condition that remains to be verified is \eqref{eq:kkt_sdp.4}. Due 
to \eqref{eq:kkt_sdps.2}, we have
\begin{align*}
0 & = \inner{Y}{\jProd{Y}{\Lambda}} = \inner{\jProd{Y}{Y}}{\Lambda}  =  \inner{G(x)}{\Lambda}. 
\end{align*}
Since $G(x)$ and $\Lambda$ are both positive semidefinite, we must have 
$\jProd{G(x)}{\Lambda} = 0$.
\end{proof}

The previous proposition leads us to consider  conditions which 
ensure that $\Lambda$ is positive semidefinite. It turns out that if the second-order sufficient condition for 
\eqref{eq:sdps} is satisfied at $(x,Y,\Lambda)$, then $\Lambda$ is positive semidefinite. 
In fact, a weaker condition is enough to ensure positive semidefiniteness.
\begin{proposition}\label{prop:kkt_slack}
Suppose that $(x,Y,\Lambda) \in \Re^n \times \S^m \times \S^m$ is a KKT  triple of \eqref{eq:sdps} such that 
$Y \in \Phi(\Lambda)$, where $\Phi(\Lambda)$ is defined by \eqref{eq:lemma:psd}, that 
is, 
\begin{equation}\label{eq:prop_kkt_slack}
	 \inner{\jProd{W}{W}}{ \Lambda } > 0
\end{equation}
for every nonzero $W \in \S^m$ such that $\jProd{Y}{W} = 0$. 
Then $(x,\Lambda)$ is a KKT pair of \eqref{eq:sdp} satisfying strict 
complementarity. 
\end{proposition}
\begin{proof}
Due to Lemma \ref{lemma:psd}, $\Lambda$ is positive semidefinite and  $\matRank  Y =  m - \matRank \Lambda$. 
Now, since $G(x) = Y^2$, we have $\matRank G(x) = \matRank Y$. Therefore $(x,\Lambda)$ must satisfy the strict 
complementarity condition.
\end{proof}

\begin{corollary}\label{col:strict}
Suppose that SOSC-NLP is satisfied at a KKT triple $(x,Y,\Lambda) \in \Re^n \times \S^m \times \S^m$. Then $(x,\Lambda)$ is a 
KKT pair for \eqref{eq:sdp} which satisfies the strict complementarity condition.
\end{corollary}
\begin{proof}
If we take $v = 0$ in the definition of SOSC-NLP, we obtain $Y \in \Phi(\Lambda)$. 
So, the result follows from Proposition \ref{prop:kkt_slack}. 
\end{proof}
The next result is a refinement of Proposition \ref{prop:sdp_to_sdps}.
\begin{proposition}
Suppose that $(x,\Lambda) \in \Re^n \times \S^m $ is a KKT pair of \eqref{eq:sdp} which satisfies the strict complementarity condition. 
Then there exists some $Y \in \Phi(\Lambda)$ such that $(x,Y,\Lambda)$ is a KKT triple of \eqref{eq:sdps}.
\end{proposition}

\begin{proof}
Without loss of generality, we may assume 
that $G(x)$ has the shape $\bigl(\begin{smallmatrix} A &0\\ 0&0 \end{smallmatrix} \bigr)$, 
where $A \in \PSDcone{k}$ and $k = \matRank G(x)$. 
Since $G(x)$ and $\Lambda$ are both positive semidefinite, 
the condition $\jProd{G(x)}{\Lambda} = 0$ is equivalent to $G(x)\Lambda = 0$.
It follows that  $\Lambda$ has the shape $\bigl(\begin{smallmatrix} 0 &0\\ 0&C \end{smallmatrix} \bigr)$
for some matrix $C\in \PSDcone{m-k}$. However, 
 strict complementarity holds only  if $C$ is positive definite. 
Therefore, it is enough to pick $Y$ to be the positive semidefinite 
matrix satisfying $Y^2 = G(x)$. 

Finally, note that if $W = \bigl(\begin{smallmatrix} W_1 & W_2\\ W_2^\T&W_3 \end{smallmatrix} \bigr)$, 
with $W \in \S^m, W_1 \in \S^k$, $W_2 \in \Re^{k\times (m-k)}$, $W_3 \in \S^{m-k}$,
 then the condition $\jProd{Y}{W} = 0$ together with Proposition \ref{prop:nonsigular} implies $W_1 = 0$, $W_2 = 0$.
 Since $C$ is positive definite, we must have $\inner{\Lambda}{\jProd{W}{W}} > 0$, if 
 $W \neq 0$. From \eqref{eq:lemma:psd}, this shows that $Y \in \Phi(\Lambda)$. 
\end{proof}

\section{Relations between constraint qualifications}\label{sec:cq}

In this section, we shall show that the nondegeneracy in Definition \ref{def:nondeg} is essentially 
equivalent to LICQ for \eqref{eq:sdps}. 
In \cite{shapiro97}, Shapiro mentions that the nondegeneracy 
condition for \eqref{eq:sdp} is an analogue of LICQ, but he 
also states that the analogy is imperfect. For instance, when $G(x)$ is diagonal, \eqref{eq:sdp} naturally 
becomes an NLP, since the semidefiniteness constraint is reduced to the nonnegativity of the diagonal elements. 
However, even in that case, LICQ and the nondegeneracy in Definition \ref{def:nondeg} might not be equivalent (see page 309 of \cite{shapiro97}). 
In this sense, it is interesting to see whether a correspondence between the conditions can be established when \eqref{eq:sdp} is 
reformulated as \eqref{eq:sdps}. Before that, 
we recall some facts about the geometry of the cone of positive semidefinite matrices. 

Let $A \in \PSDcone{m}$ and let $U$ be an $m\times k$ matrix whose columns form a 
basis for the kernel of $A$. Then, the tangent cone of $\PSDcone{m}$ at $A$ is 
written as
\begin{displaymath}
  \tanCone{A}{\PSDcone{m}} = \{E \in \S^m \mid U^\T EU \in \PSDcone{k}\}
\end{displaymath}
(see \cite{pataki_handbook} or \cite[Equation~26]{shapiro97}).
For example, if $A$ can be written as $\bigl(\begin{smallmatrix} D & 0 \\ 0 & 0 \end{smallmatrix} \bigr)$, where 
$D$ is positive definite, then the matrices in $\tanCone{A}{\PSDcone{m}}$ have 
the shape $\bigl(\begin{smallmatrix} C & F \\ F^\T & H \end{smallmatrix} \bigr)$, where 
the only restriction is that $H$ should be positive semidefinite.

Our first step is to notice that nondegeneracy implies 
that the only matrix which is orthogonal to both $\lineality \tanCone{G(x)}{\PSDcone{m}}$ 
and $\mathrm{Im}\, \grad G(x) $ is the trivial one, i.e., 
\begin{equation}
W \in (\lineality \tanCone{G(x)}{\PSDcone{m}})^\perp \text{ and } \grad G(x)^*W = 0 \quad \implies \quad W = 0, \tag{Nondegeneracy}\label{eq_nondegeneracy2}
\end{equation}
where $\perp$ denotes the orthogonal complement.

On the other hand, the LICQ constraint 
qualification for \eqref{eq:sdps} holds at a feasible point $(x,Y)$ if the linear function which 
maps $(v,W)$ to $\grad G(x)v -2\jProd{W}{Y}$ is surjective. This happens 
if and only if the adjoint map has trivial kernel. The adjoint map 
takes $W \in \S^m$ and maps it to $(\grad G(x)^*W, -2\jProd{W}{Y})$. So
the surjectivity assumption amounts to requiring that 
every $W$ which satisfies both $\grad G(x)^*W = 0 $ and $\jProd{W}{Y} = 0 $ must actually 
be $0$, that is,
\begin{equation}
\jProd{W}{Y} = 0 \text{ and } \grad G(x)^*W = 0 \quad \implies \quad  W = 0. \tag{LICQ}\label{eq_licq}
\end{equation} 
The subspaces $\ker L_Y = \{W \mid \jProd{Y}{W} = 0 \}$ and 
$(\lineality \tanCone{G(x)}{\PSDcone{m}})^\perp$ are closely related. The 
next proposition clarifies this connection.

\begin{proposition}\label{prop:nondeg_licq}
Let $V = Y^2$, then $(\lineality \tanCone{V}{\PSDcone{m}})^\perp \subseteq \ker L_Y$. 
If $Y$ is positive semidefinite, then $\ker L_Y \subseteq (\lineality \tanCone{V}{\PSDcone{m}})^\perp$ as 
well.
\end{proposition}
\begin{proof}
Note that if $Q$ is an orthogonal matrix, then $\tanCone{Q^\T VQ}{\PSDcone{m}} = Q^\T \tanCone{V}{\PSDcone{m}} Q$. 
The same is true for $\ker L_Y$, i.e., $\ker L_{Q^\T YQ} = Q^\T \ker L_YQ $. So, without loss 
of generality, we may assume that $Y$ is diagonal and that
\begin{equation*}
Y = \begin{pmatrix}
D & 0 \\ 0 & 0
\end{pmatrix},
\end{equation*}
where $D$ is an $r\times r$ nonsingular  diagonal matrix. Then, we have  
\begin{align*}
\tanCone{V}{\PSDcone{m}} &= \left\{\begin{pmatrix}A & B \\ B^\T & C \end{pmatrix} \, \middle| \, A \in \S^{r}, B \in \Re^{r\times (m-r)},   C \in \PSDcone{m-r}    \right\}, \\
\lineality \tanCone{V}{\PSDcone{m}} &= \left\{\begin{pmatrix}A & B \\ B^\T & 0 \end{pmatrix} \, \middle| \, A \in \S^{r}, B \in \Re^{r\times (m-r)}     \right\},\\
(\lineality \tanCone{V}{\PSDcone{m}})^\perp &= \left\{\begin{pmatrix}0 & 0 \\ 0 & C \end{pmatrix} \, \middle| \, C \in \S^{m-r}   \right\}.
\end{align*}
This shows that every matrix $Z \in (\lineality \tanCone{V}{\PSDcone{m}})^\perp$ satisfies 
$YZ = 0$ and therefore lies in $\ker L_Y$. Now, the kernel of $L_Y$ can be described as follows:
\begin{equation*}
\ker L_Y = \left\{\begin{pmatrix}A & 0 \\ 0 & C \end{pmatrix} \, \middle| \, \jProd{A}{D} = 0,  C \in \S^{m-r}   \right\}.
\end{equation*}
If $Y$ is positive semidefinite, then $D$ is positive definite and the operator $L_D$ is nonsingular. Hence 
$\jProd{A}{D} = 0 $ implies $A = 0$. In this case, $\ker L_Y$ coincides with 
$(\lineality \tanCone{V}{\PSDcone{m}})^\perp$.
\end{proof}

\begin{corollary}\label{col:licq_nondeg}
If $(x,Y) \in \Re^n \times \S^m$ satisfies LICQ for the problem \eqref{eq:sdps}, then nondegeneracy is 
satisfied at $x$ for \eqref{eq:sdp}. On the other hand, if $x$ satisfies nondegeneracy
and if $Y = \sqrt{G(x)}$, then $(x,Y)$ satisfies LICQ for \eqref{eq:sdps}.
\end{corollary}

\begin{proof}
  It follows easily from Proposition~\ref{prop:nondeg_licq}.
\end{proof}


\section{Analysis of second-order sufficient conditions}\label{sec:sosc}
In this section, we examine the relations between KKT points 
of \eqref{eq:sdp} and \eqref{eq:sdps} that satisfy second-order 
sufficient conditions.

\begin{proposition}\label{prop:sosc_sdps_to_sdp}
Suppose that $(x,Y,\Lambda) \in \Re^n \times \S^m \times \S^m$ is a KKT triple of \eqref{eq:sdps} satisfying 
SOSC-NLP. Then $(x,\Lambda)$ is a KKT pair of \eqref{eq:sdp} that satisfies 
strict complementarity and \eqref{eq:sosc_sdp.1}. If additionally, $(x,Y,\Lambda)$ satisfies LICQ for \eqref{eq:sdps} or 
$(x,\Lambda)$ satisfies nondegeneracy for \eqref{eq:sdp}, then $(x,\Lambda)$ satisfies SOSC-SDP as well.
\end{proposition}
\begin{proof}
In Corollary \ref{col:strict}, we have already shown that $(x,\Lambda)$ is a KKT 
pair of \eqref{eq:sdp} and strict complementarity is satisfied. In addition, 
if $(x,Y,\Lambda)$ satisfies LICQ for \eqref{eq:sdps}, then Corollary \ref{col:licq_nondeg} ensures 
that $(x,\Lambda)$ satisfies nondegeneracy for \eqref{eq:sdp}. It only 
remains to  show that 
\eqref{eq:sosc_sdp.1} is also satisfied. 
To this end, consider an arbitrary nonzero $d \in \Re^n$ such that  $\grad G(x)d \in \tanCone{G(x)}{\PSDcone{m}}$ and 
$\inner{\grad f(x)} { d} = 0$. We are thus required to show that
\begin{equation}\label{eq_objective}
\inner{(\grad _x ^2 L(x,\Lambda) + H(x,\Lambda))d}{d} > 0,
\end{equation}
where $H(x,\Lambda)$ is defined in \eqref{eq:sosc_sdp.2}.
A first observation is that, due to \eqref{eq:kkt_sdp.1}, 
we have $\inner{\grad G(x)d}{\Lambda}  = \inner{d}{\grad G(x)^*\Lambda} = \inner{d} {\grad f(x)}= 0$, 
that is, $\grad G(x)d \in \{\Lambda\}^\perp$.
We recall that $H(x,\Lambda)$ satisfies
\begin{equation}\label{eq:hessian}
\inner{H(x,\Lambda)d}{d} = 2\inner{\Lambda}{ (\grad G(x)d)^\T G^\dagger(x) \grad G(x)d}.
\end{equation}
The strategy here is to first identify the shape and properties of 
  several matrices involved, before showing that \eqref{eq_objective} holds.
Without loss of generality, we may assume that $G(x)$ is diagonal, i.e.,
\begin{equation*}
G(x) = \begin{pmatrix} D & 0 \\ 0 & 0 \end{pmatrix},
\end{equation*}
where $D$ is a $k\times k$ diagonal  positive definite  matrix. 
We also have
\begin{equation*}
Y = \begin{pmatrix} E & 0 \\ 0 & 0 \end{pmatrix},
\end{equation*}
with $E$ an invertible diagonal matrix such that $E^2 = D$. Since SOSC-NLP holds, considering $v = 0$ in 
\eqref{eq:sosc_nlp.1}, we obtain $\inner{\jProd{W}{W}}{\Lambda} > 0$ for all nonzero $W \in \S^m$ such 
that $\jProd{W}{Y} = 0$,  which, by \eqref{eq:lemma:psd}, shows that $Y \in \Phi(\Lambda)$.
From \eqref{eq:kkt_sdps.2}, we also have $\jProd{Y}{\Lambda} = 0$. 
Thus, Lemma \ref{lemma:psd} ensures that every pair $\sigma, \sigma'$ of nonzero eigenvalues of $Y$ 
satisfies $\sigma + \sigma'\neq 0$. Since the eigenvalues of $E$ are precisely the 
nonzero eigenvalues of $Y$, it follows that $L_E$ is an invertible operator, 
by virtue of Proposition \ref{prop:nonsigular}. 
Moreover, due to the strict  complementarity condition, we obtain
\begin{equation*}
\Lambda = \begin{pmatrix} 0 & 0 \\ 0 & \Gamma \end{pmatrix},
\end{equation*}
where $\Gamma \in \PSDcone{m-k}$ is positive definite. The pseudo-inverse of $G(x)$ is given by
\begin{equation*}
G(x)^\dagger = \begin{pmatrix} D^{-1} & 0 \\ 0 & 0 \end{pmatrix}.
\end{equation*}
We partition $\grad G(x)d$ in blocks in the following fashion:
\begin{equation*}
\grad G(x)d = \begin{pmatrix} A & B \\ B^\T & C \end{pmatrix},
\end{equation*}
where $A \in \S^k$, $B\in \Re^{k\times (m-k)}$ and $C \in \S^{m-k}$.
Inasmuch as $\grad G(x)d$ lies in the tangent cone $\tanCone{G(x)}{\PSDcone{m}}$, $C$ must be positive semidefinite. 
However, as observed earlier, we have $\inner{\grad G(x)d}{\Lambda} = 0$, which yields $\inner{C}{\Gamma} = 0$. 
Since $\Gamma$ is positive definite, this implies $C = 0$, and hence 
\begin{equation*}
\grad G(x)d = \begin{pmatrix} A & B \\ B^\T & 0 \end{pmatrix}.
\end{equation*}
We are now ready to show that \eqref{eq_objective} holds. We shall 
do that by considering $v = d$ in \eqref{eq:sosc_nlp.1} and exhibiting some  $W$ such that 
$\grad G(x)d - 2\jProd{Y}{W} = 0$ and $2\inner{\jProd{W}{W}}{\Lambda} = \inner{H(x,\Lambda)d}{d}$.
Then  SOSC-NLP  will ensure that \eqref{eq_objective} holds. Note 
that, for any $Z\in \S^{m-k} $, 
\begin{equation*}
W_Z = \begin{pmatrix}L_E^{-1} (A)/2 & E^{-1}B \\ B^\T E^{-1} & Z \end{pmatrix}
\end{equation*}
is a solution to the equation $\grad G(x)d - 2\jProd{Y}{W} = 0$. Moreover, any solution to that 
equation must have this particular shape. 
Therefore, the proof will be complete if we can choose $Z$ such that  $2\inner{\jProd{W}{W}}{\Lambda} = \inner{H(x,\Lambda)d}{d}$ holds.
Observe that, by \eqref{eq:hessian},
\begin{align}
2\inner{\jProd{W_Z}{W_Z}}{\Lambda} -\inner{H(x,\Lambda)d}{d} & = 2\inner{W_Z^2}{\Lambda} - 2\inner{ (\grad G(x)d)^\T G^\dagger(x) \grad G(x)d}{\Lambda} \nonumber\\
& = 2\inner{Z^2 + B^\T E^{-2}B }{\Gamma} - 2 \inner{B^\T D^{-1}B }{\Gamma} \nonumber \\
& = \inner{2Z^2 + 2B^\T D^{-1}B - 2B^\T D^{-1}B}{\Gamma} \nonumber\\
 & = \inner{2Z^2 }{\Gamma}. \label{eq_non_negative}
\end{align}
Thus, taking $Z = 0$ yields $2\inner{\jProd{W_Z}{W_Z}}{\Lambda} = \inner{H(x,\Lambda)d}{d}$. 
This completes the proof.
\end{proof}

\begin{proposition}\label{prop:sosc_sdp_to_sdps}
Suppose that $(x,\Lambda) \in \Re^n \times \S^m$ is a KKT pair for \eqref{eq:sdp} 
satisfying~\eqref{eq:sosc_sdp.1} and the strict complementarity condition. Then, there exists $Y \in \S^m$ such 
that $(x,Y,\Lambda)$ is a KKT triple for \eqref{eq:sdps} satisfying SOSC-NLP.
\end{proposition}
\begin{proof}
Again, we assume without loss of generality that $G(x)$ is diagonal, so that 
\begin{equation*}
G(x) = \begin{pmatrix} D & 0 \\ 0 & 0 \end{pmatrix},
\end{equation*}
where $D$ is a $k\times k$ diagonal  positive definite matrix. 
Take $Y$ such that 
\begin{equation*}
Y = \begin{pmatrix} E & 0 \\ 0 & 0 \end{pmatrix},
\end{equation*}
where $E^2=D$ and $E$ is positive definite; in particular $L_E$ is invertible. Then $(x,Y,\Lambda)$ is a 
KKT triple for \eqref{eq:sdps}. Due to strict complementarity, we have 
\begin{equation*}
\Lambda = \begin{pmatrix} 0 & 0 \\ 0 & \Gamma \end{pmatrix},
\end{equation*}
where $\Gamma \in \PSDcone{m-k}$ is positive definite.
We are required to show that
\begin{equation}\label{eq_objective2}
\inner{\grad _x ^2 L(x,\Lambda)v}{v} + 2\inner{\jProd{W}{W}}{\Lambda} > 0
\end{equation}
for every nonzero $(v,W)$ such that $\grad G(x)v - 2\jProd{Y}{W} = 0$. 
So, let 
$(v,W)$ satisfy $\grad G(x)v - 2\jProd{Y}{W} = 0$. Let us first consider 
what happens when $v = 0$. Partitioning $W$ in blocks, we have
\begin{equation*}
\jProd{Y}{ \begin{pmatrix}W_1 & W_2 \\ W_2^\T & W_3 \end{pmatrix}} =  \begin{pmatrix}\jProd{E}{W_1} & EW_2/2 \\ W_2^\T E/2 & 0 \end{pmatrix}.
\end{equation*}
Recall that $L_E$ as well as $E$ is invertible. So, $\jProd{Y}{W} = 0$ implies $W_1 = 0$  and $W_2 = 0$. 
If $W\neq 0$, then $W_3$ must be nonzero, which in turn implies that $\jProd{W_3}{W_3}$ must also be nonzero.  
We then have $\inner{\jProd{W}{W}}{\Lambda} = \inner{\jProd{W_3}{W_3}}{\Gamma}$.
But $\inner{\jProd{W_3}{W_3}}{\Gamma} $ must be greater than zero, since $\Gamma$ is positive definite.
Thus, in this case, \eqref{eq_objective2} is satisfied.

Now, we suppose that $v$ is nonzero. First, 
we will show that $\grad G(x)v$ lies in the tangent cone $\tanCone{G(x)}{\PSDcone{m}} $  and that $\grad G(x)v$ 
is orthogonal to $\Lambda$, which implies $0 = \inner{\grad G(x)v}{\Lambda} = v^\T\grad G(x)^*\Lambda = v^\T \grad f(x)$. This shows that $v$ lies in the critical cone $\mathcal{C}(x)$.

Note that the image of the operator $L_Y$ only contains matrices having the 
lower right $(m-k)\times (m-k)$ block equal to zero. Therefore, $\grad G(x)v = 2\jProd{Y}{W}$ implies 
that $\grad G(x)v$ has the shape
\begin{equation*}
\grad G(x)v = \begin{pmatrix} A & B \\ B^\T & 0 \end{pmatrix}.
\end{equation*}
Hence, $\grad G(x)v \in \tanCone{G(x)}{\PSDcone{m}} $ and $\grad G(x)v$ is orthogonal to $\Lambda$. Due to SOSC-SDP, 
we must have 
\begin{equation*}
\inner{(\grad _x ^2 L(x,\Lambda) + H(x,\Lambda))v}{v} > 0.
\end{equation*}
Thus, if  $\inner{ H(x,\Lambda)v}{v} \leq 2\inner{\jProd{W}{W}}{\Lambda}$ holds, 
then we have \eqref{eq_objective2}. In fact, since $W$ satisfies 
$\grad G(x)v - 2\jProd{Y}{W} = 0$, the chain of equalities finishing at 
\eqref{eq_non_negative} readily yields $\inner{ H(x,\Lambda)v}{v} \leq 2\inner{\jProd{W}{W}}{\Lambda}$.
\end{proof}

Here, we remark one interesting consequence of the previous analysis. The second-order 
\emph{sufficient} condition for NSDPs in \cite{shapiro97} is stated under the assumption that the 
pair $(x,\Lambda)$ satisfies both strict complementarity and nondegeneracy. However, since 
\eqref{eq:sdp} and \eqref{eq:sdps} share the same local minima, 
Proposition \ref{prop:sosc_sdp_to_sdps} implies that we may remove the nondegeneracy assumption from 
SOSC-SDP. We now state a sufficient condition for \eqref{eq:sdp} based on the analysis above.

\begin{proposition}[A Sufficient Condition via Slack Variables]\label{prop:sosc_slack}
Let $(x,\Lambda) \in \Re^n \times \S^m$ be a KKT pair for \eqref{eq:sdp} satisfying 
strict complementarity. Assume also that the following condition holds:
\begin{equation}
  \inner{\grad _x ^2 L(x,\Lambda)v}{v} + 2\inner{\jProd{W}{W}}{\Lambda} > 0
\end{equation}
for every nonzero $(v,W) \in \Re^n \times \S^m$ such that $\grad G(x)v - 2\jProd{\sqrt{G(x)}}{W} = 0$.
Then, $x$~is a local minimum for \eqref{eq:sdp}.
\end{proposition}
Apart from the detail of requiring nondegeneracy, the condition above is 
equivalent to SOSC-SDP, due to Propositions \ref{prop:sosc_sdps_to_sdp} and \ref{prop:sosc_sdp_to_sdps}.

\section{Analysis of second-order necessary conditions}\label{sec:sonc}

We now take a look at the difference between second-order necessary conditions 
that can be derived from \eqref{eq:sdp} and \eqref{eq:sdps}. Since the inequalities 
\eqref{eq:sonc_nlp} and \eqref{eq:sonc_sdp} are not strict, we need a slightly 
stronger assumption to prove the next proposition.

\begin{proposition}\label{prop:sonc_sdps_to_sdp}
Suppose that $(x,Y,\Lambda) \in \Re^n \times \S^m \times \S^m$ is a KKT triple of \eqref{eq:sdps} satisfying 
LICQ and SONC-NLP. Moreover, assume that $Y$ and $\Lambda$ are positive semidefinite.  If $(x,\Lambda)$ 
is a KKT pair for \eqref{eq:sdp} satisfying strict complementarity, then it also satisfies SONC-SDP.
\end{proposition}

\begin{proof}
Since $(x,Y,\Lambda)$ satisfies LICQ for \eqref{eq:sdps} and $Y$ is positive semidefinite, 
Corollary \ref{col:licq_nondeg} implies that $(x,\Lambda)$ satisfies nondegeneracy. 
Under the assumption that $(x,\Lambda)$ is strictly complementary, the only thing missing 
is to show that \eqref{eq:sonc_sdp} holds. To do so, we proceed as in Proposition 
\ref{prop:sosc_sdps_to_sdp}. We partition $G(x),Y, \Lambda$, $G(x)^\dagger$ and $\nabla G(x)d$ in blocks in 
exactly the same way. The only difference is that,  since \eqref{eq:sonc_nlp} 
does not hold strictly, we cannot make use of Lemma \ref{lemma:psd} in order to conclude 
that $L_E$ is invertible. Nevertheless, since we assume
that $Y$ is positive semidefinite, all the eigenvalues of $E$ are strictly positive anyway. So, as before, 
we can conclude that $L_E$ is an invertible operator, by Proposition \ref{prop:nonsigular}. 
Due to 
strict complementarity, we can also conclude that $\Gamma \in \PSDcone{m-k}$ is positive definite and 
that $C = 0$.  

All our ingredients are now 
in place and we can proceed exactly as in the proof of Proposition \ref{prop:sosc_sdps_to_sdp}. Namely, 
we have to prove that, given $d \in \mathcal{C}(x)$, the inequality 
$\inner{(\grad_x^2L(x,\Lambda) + H(x,\Lambda))d}{d} \geq 0$ holds. As before, 
the way to go is to craft a matrix $W$ satisfying both $\nabla G(x)d -2\jProd{Y}{W} = 0$ 
and $\inner{H(x,\Lambda)d}{d}  = 2\inner{\jProd{W}{W}}{Z} $. Then SONC-NLP will ensure 
that \eqref{eq:sonc_sdp} holds. Actually, it can be done by taking 
\begin{equation*}
W  = \begin{pmatrix}L_E^{-1} (A)/2 & E^{-1}B \\ B^\T E^{-1} & 0 \end{pmatrix}
\end{equation*}
and following the same line of arguments that leads to \eqref{eq_non_negative}.
\end{proof}

\begin{proposition}\label{prop:sonc_sdp_to_sdps}
Suppose that $(x,\Lambda) \in \Re^n \times \S^m$ is a KKT pair for \eqref{eq:sdp} satisfying 
SONC-SDP. Then, there exists $Y \in \S^m$ such 
that $(x,Y,\Lambda)$ is a KKT triple for \eqref{eq:sdps} satisfying  SONC-NLP.
\end{proposition}

\begin{proof}
It is enough to choose $Y$ to be $\sqrt{G(x)}$. If 
we do so, Corollary \ref{col:licq_nondeg} ensures that $(x,Y,\Lambda)$ satisfies 
LICQ. We now have to check that \eqref{eq:sonc_nlp} holds. For this,  
we can follow the proof of Proposition \ref{prop:sosc_sdp_to_sdps} by considering 
  \eqref{eq:sonc_sdp} instead of \eqref{eq:sosc_sdp.1}. No special 
considerations are needed for this case. 
\end{proof}

Assume that $(x,\Lambda)$ is a KKT pair of \eqref{eq:sdp} satisfying nondegeneracy and 
strict complementarity. Then, Proposition \ref{prop:sonc_sdps_to_sdp} gives an elementary route to prove 
that SONC-SDP  holds. This is because if we select $Y$ to be 
the positive semidefinite square root of $G(x)$, all the conditions of Proposition \ref{prop:sonc_sdps_to_sdp} are 
satisfied, which means that \eqref{eq:sonc_sdp} must hold. Moreover, if we were to derive second-order necessary 
conditions for \eqref{eq:sdp} from scratch, we could consider the following.

\begin{proposition}[A Necessary Condition via Slack Variables]\label{prop:sonc_slack}
Let $x \in \Re^n$ be a local minimum of \eqref{eq:sdp}.
Assume that $(x,\Lambda) \in \Re^n \times \S^m$ is a KKT pair for \eqref{eq:sdp} satisfying strict complementarity and
nondegeneracy. Then the following condition holds:
\begin{equation}
    \inner{\grad _x ^2 L(x,\Lambda)v}{v} + 2\inner{\jProd{W}{W}}{\Lambda} \geq 0 
\end{equation}
 for every $(v,W) \in \Re^n \times \S^m$ such that $\grad G(x)v - 2\jProd{\sqrt{G(x)}}{W} = 0$.
\end{proposition} 
Propositions \ref{prop:sonc_sdps_to_sdp} and \ref{prop:sonc_sdp_to_sdps} ensure that the condition 
above is equivalent to SONC-SDP. 
Comparing Propositions \ref{prop:sosc_slack} and \ref{prop:sonc_slack}, we see that the second-order 
conditions derived through the aid of slack variables have ``no-gap'' in the sense 
that, apart from regularity conditions, the only difference between them is the change 
from ``$>$'' to ``$\geq$''.


\section{Computational experiments}
\label{sec:comp}


Let us now examine the validity of the squared slack variables method for NSDP problems.
We tested the slack variables approach on a few simple problems. Our solver of 
choice was PENLAB \cite{Penlab}, which is based on PENNON \cite{Pennon} and uses an algorithm 
based on the augmented Lagrangian technique. As far as we know, 
PENLAB is the only open-source general nonlinear programming solver capable of handling 
nonlinear SDP constraints. Because of that, we have the chance of comparing the ``native'' approach 
against the slack variables approach using the same code. We ran PENLAB with the default parameters.
All the tests were done on a 
notebook with the following specs: Ubuntu 14.04, CPU Intel i7-4510U with 4 cores operating at 
2.0Ghz and 4GB of RAM.

In order to use an NLP solver to tackle \eqref{eq:sdp}, we have to select a 
vectorization strategy. We decided to vectorize an $n\times n$ symmetric matrix 
by transforming it into an $\frac{n(n+1)}{2}$ vector, such that the columns of 
the lower triangular part are stacked one on top of the other. For instance,  
the matrix $\bigl(\begin{smallmatrix} 1 & 2 \\ 2 & 3 \end{smallmatrix} \bigr)$ is transformed 
to the column vector $(1,2,3)^\T$.

\subsection{Modified Hock-Schittkowski problem 71}

There is a known suite of problems for testing nonlinear optimization 
problems collected by Hock and Schittkowski \cite{HS80,S09}. The problem below is a modification 
of problem 71 of \cite{S09} and it comes together with PENLAB. Both the constraints 
and the objective function are nonconvex.
 The problem has the following formulation:
\begin{equation}
  \label{eq:Hock}
  \tag{HS}
  \begin{array}{ll}
    \underset{x \in \Re^6}{\mbox{minimize}} & x_1x_4(x_1+x_2+x_3) + x_3 \\ 
    \mbox{subject to } & x_1x_2x_3x_4 - x_5 -25  = 0,\\[5pt]
    & x_1^2 + x_2^2 + x_3^2 + x_4^4 - x_6 - 40 = 0,\\[5pt]
    & \begin{pmatrix}x_1 & x_2 & 0 & 0 \\ x_2 & x_4 & x_2+x_3 & 0 \\
    0 & x_2+x_3 & x_4 & x_3 \\
    0 & 0 & x_3 & x_1
    \end{pmatrix}  \in \PSDcone{4},\\[20pt]
    & 1 \leq x_i \leq 5, \qquad i = 1,2,3,4 \\[5pt]
    & x_i \geq 0, \qquad i = 5,6.
  \end{array}
\end{equation}

We reformulate the problem \eqref{eq:Hock} by removing the positive 
semidefiniteness constraints  and adding a squared slack variable $Y$. We then 
test both formulations using PENLAB. The initial point is set 
to be $x = (5,5,5,5,0,0)$ and the slack variable to be the identity matrix $Y = I_4$. 
This produces infeasible points for 
both formulations. Nevertheless, PENLAB was able to solve the problem via both 
approaches. The results can be seen in Table \ref{table:hock}. The first three columns 
count the numbers of evaluations of the augmented Lagrangian function, its gradients and its Hessians, 
respectively. The fourth column is the number of outer iterations. The ``time'' column indicates 
the time in seconds as measured by PENLAB. The last column indicates the optimal 
value obtained. It seems that there were no significant differences in performance between 
both approaches.

\begin{table}
\centering
\caption{Slack vs ``native'' for \eqref{eq:Hock}}
\label{table:hock}
\begin{tabular}{@{}lcccccc@{}}
\toprule
       & functions & gradients & Hessians & iterations & time (s) & opt. value \\ \midrule
slack  & 110      & 57        & 44       & 13         & 0.54    & 87.7105 \\
native & 123      & 71        & 58       & 13         & 0.57    & 87.7105 \\ \bottomrule
\end{tabular}
\end{table}

\subsection{The closest correlation matrix problem --- simple version}\label{sec:ccm1}

Given an $m\times m$ symmetric matrix $H$ with diagonal entries equal to one, 
we want to find the element in 
$\PSDcone{m}$ which is closest to $H$ and 
has all  diagonal 
entries  also equal to one. The problem can be formulated as follows:

\begin{equation}
  \label{eq:cor1}
  \tag{Cor}
  \begin{array}{ll}
    \underset{X}{\mbox{minimize}} & \inner{X - H}{X-H} \\ 
    \mbox{subject to} & X_{ii} = 1 \qquad \forall \, i,\\
    & X  \in \PSDcone{m}.
  \end{array}
\end{equation}
This problem is convex and, due to its structure, we can use slack variables without increasing 
the number of variables. We have the following formulation:
\begin{equation}
  \label{eq:cor2}
  \tag{Cor-Slack}
  \begin{array}{ll}
    \underset{X}{\mbox{minimize}} & \inner{(\jProd{X}{X}) - H}{(\jProd{X}{X})- H} \\ 
    \mbox{subject to} & (\jProd{X}{X})_{ii} = 1 \qquad \forall \, i, \\
    & X  \in S^m.
  \end{array}
\end{equation}

In our experiments, we generated 100 symmetric matrices $H$ such that the diagonal 
elements are all $1$ and  other elements are uniform random numbers between~$-1$ and $1$. 
For both \ref{eq:cor1} and \ref{eq:cor2}, we used $X = I_m$ as an initial solution in all instances.
We solved problems with $m = 5,10,15,20$ and the results can be found in Table \ref{table:ex1}.
The columns ``mean'', ``min'' and ``max'' indicate, respectively,  
the mean, minimum and maximum of the running times in seconds of all instances.
For this problem, both formulations were able to solve all instances. 
We included the ``mean time'' column just to give an idea about the magnitude of the running time. In reality,
for fixed $m$, the running time oscillated highly among different instances, as  can be seen 
by the difference between the maximum and the minimum running times. 
We noted no significant difference between the optimal values obtained from both formulations.

\begin{table}[!htb]
\centering
\caption{Comparison between \ref{eq:cor1} and \ref{eq:cor2} }
\label{table:ex1}
\begin{tabular}{@{}lccc|ccc@{}}
\toprule
   & \multicolumn{3}{c}{\ref{eq:cor2} }    & \multicolumn{3}{c}{\ref{eq:cor1}}   \\ \midrule
$m$ & mean (s) & min (s) & max (s) & mean (s) & min (s) & max (s) \\ \midrule
5  & 0.090    & 0.060   & 0.140   & 0.201    & 0.130   & 0.250   \\
10 & 0.153    & 0.120   & 0.230   & 0.423    & 0.330   & 0.630   \\
15 & 0.287    & 0.210   & 0.430   & 1.306    & 1.020   & 1.950   \\
20 & 0.556    & 0.450   & 1.180   & 3.491    & 2.820   & 4.990    \\ \bottomrule
\end{tabular}
\end{table}

We tried, as much as possible, to 
implement gradients and Hessians of both problems in a similar way. As \ref{eq:cor1} is an 
example that comes with PENLAB, we also performed some minor tweaks to conform 
to that goal. Performance-wise, the formulation \ref{eq:cor2} seems to be competitive for this example. In most 
instances, \ref{eq:cor2} had a faster running time. In Figure \ref{fig:corr_20}, we show the comparison 
between running times, instance-by-instance,  for the case $m = 20$.

\begin{figure}
\centering
\includegraphics[scale=0.7]{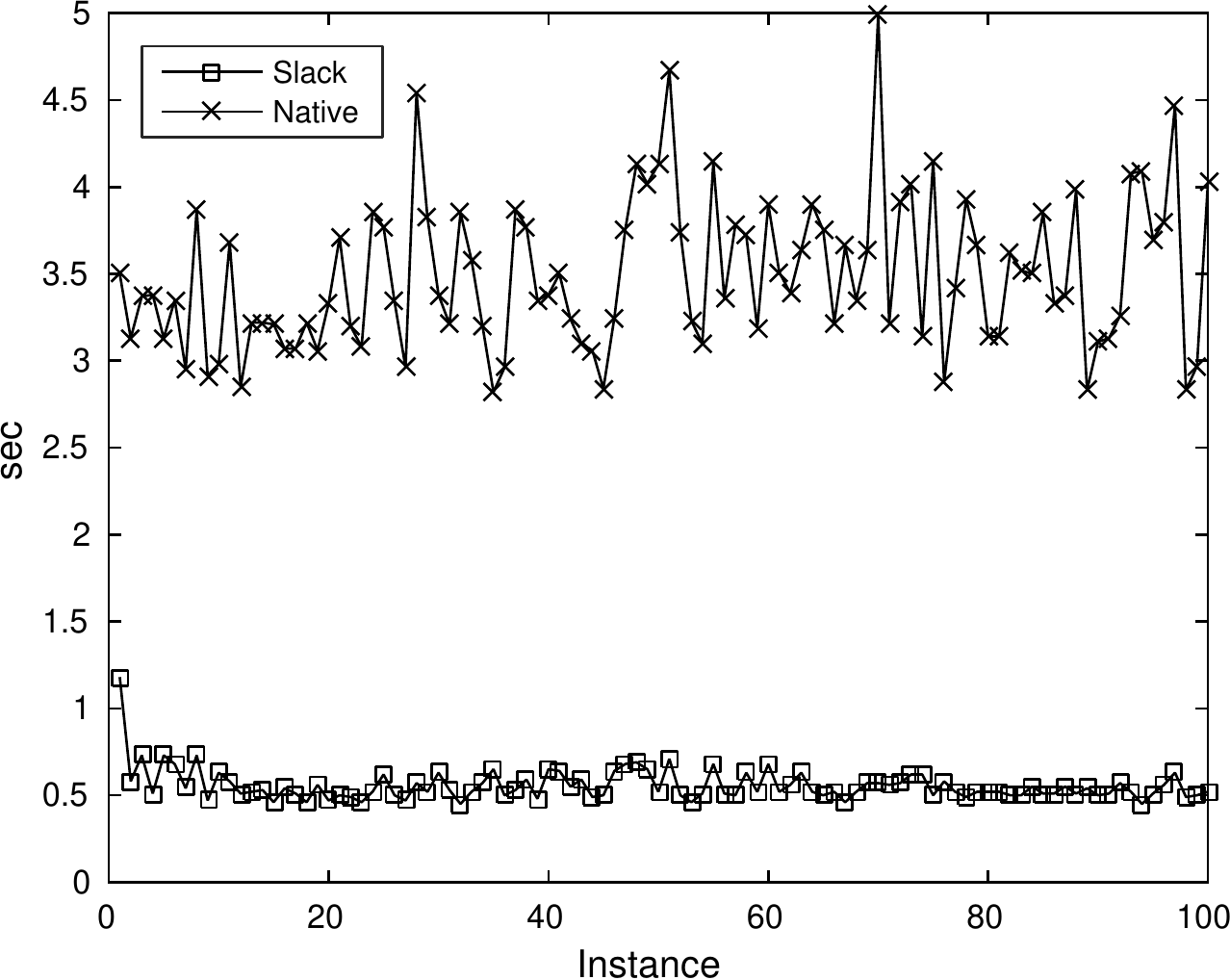}\caption{\ref{eq:cor1} vs. \ref{eq:cor2}. Instance-by-instance running times for $m = 20$.}\label{fig:corr_20}
\end{figure}

\subsection{The closest correlation matrix problem --- extended version}\label{sec:ccm2}
 We consider an extended formulation for \ref{eq:cor1} as suggested in 
 one of PENLAB's examples, with extra constraints to bound the eigenvalues 
 of the matrices:
 
 \begin{equation}
  \label{eq:cor_ext1}
  \tag{Cor-Ext}
  \begin{array}{ll}
    \underset{X,z}{\mbox{minimize }} & \inner{zX - H}{zX-H} \\ 
    \mbox{subject to } & zX_{ii} = 1 \qquad \forall \, i,\\
    &  I_m \preceq X \preceq \kappa I_m, 
  \end{array}
\end{equation}
where  $\kappa$ is some positive number greater than $1$ and 
the notation $X \succeq \kappa I_m $ means $X - \kappa I_m \in \PSDcone{m}$. This is 
a nonconvex problem, and using slack 
variables, we obtain the following formulation:
 \begin{equation}
  \label{eq:cor_ext2}
  \tag{Cor-Ext-Slack}
  \begin{array}{l}
    \begin{array}{ll}
      \underset{X,Y_1,Y_2,z}{\mbox{minimize }} & \inner{zX - H}{zX-H} \\ 
    \end{array}
    \\
    \begin{array}{llll}
      \mbox{subject to } & zX_{ii} & = 1 & \forall \, i,\\
      &   \kappa I_m - X & = \jProd{Y_1}{Y_1}, &\\ 
      & X - I_m & = \jProd{Y_2}{Y_2}. &
    \end{array}
  \end{array}
\end{equation}

In our experiments, we set $\kappa = 10$.
As before, we generated 100 symmetric matrices $H$ whose diagonal 
elements are all $1$ and  other elements are uniform random numbers between $-1$ and $1$. 
For \ref{eq:cor_ext1}, we used $z = 1$ and $X = I_m$ as initial points.
For \ref{eq:cor_ext2}, we used an infeasible starting point $z = 1$, $X = Y_2 = I_m$ and 
$Y_1 = 3I_m$.  We solved problems with $m = 5,10,15,20$ and the results can be found in Table \ref{table:ex2}.
The columns have the same meaning as in Section\ref{sec:ccm1}. This time, we saw 
a higher failure rate for the  formulation \ref{eq:cor_ext2}.  We tried a few different 
initial points, but the results stayed mostly the same.
The best results were
obtained for the case $m = 5$ and $m = 10$, where \ref{eq:cor_ext2} had a performance comparable to 
\ref{eq:cor_ext1}, although the latter seldom failed. For $m = 15$ and $m = 20$, 
\ref{eq:cor_ext2} was slower than \ref{eq:cor_ext1}, which is expected, because the number of 
variables increased significantly. However, it was still able 
to solve the majority of instances. In Figure \ref{fig:corr_ext}, we show the comparison 
of running times, instance-by-instance,  for the cases $m = 10$ and $m = 20$.

\begin{table}
\centering
\caption{Comparison between \ref{eq:cor_ext1} and \ref{eq:cor_ext2}}
\label{table:ex2}
\begin{tabular}{@{}lcccc|cccc@{}}
\toprule
   & \multicolumn{4}{c}{\ref{eq:cor_ext2}}           & \multicolumn{4}{c}{\ref{eq:cor_ext1}}          \\ \midrule
$m$   & mean (s) & min (s) & max (s) & fail & mean (s) & min (s) & max (s) & fail \\ \midrule
5  & 0.236    & 0.130   & 0.830   & 15   & 0.445    & 0.250   & 2.130   & 1    \\
10 & 0.741    & 0.420   & 2.580   & 3    & 1.206    & 0.580   & 7.300   & 0    \\
15 & 4.651    & 2.090   & 26.96   & 15   & 3.809    & 1.960   & 14.12   & 0    \\
20 & 24.32    & 15.20   & 69.34   & 8    & 9.288    & 5.150   & 36.81   & 0    \\ \bottomrule
\end{tabular}
\end{table}

\begin{figure}
\centering
\begin{subfigure}{}
\includegraphics[scale=0.7]{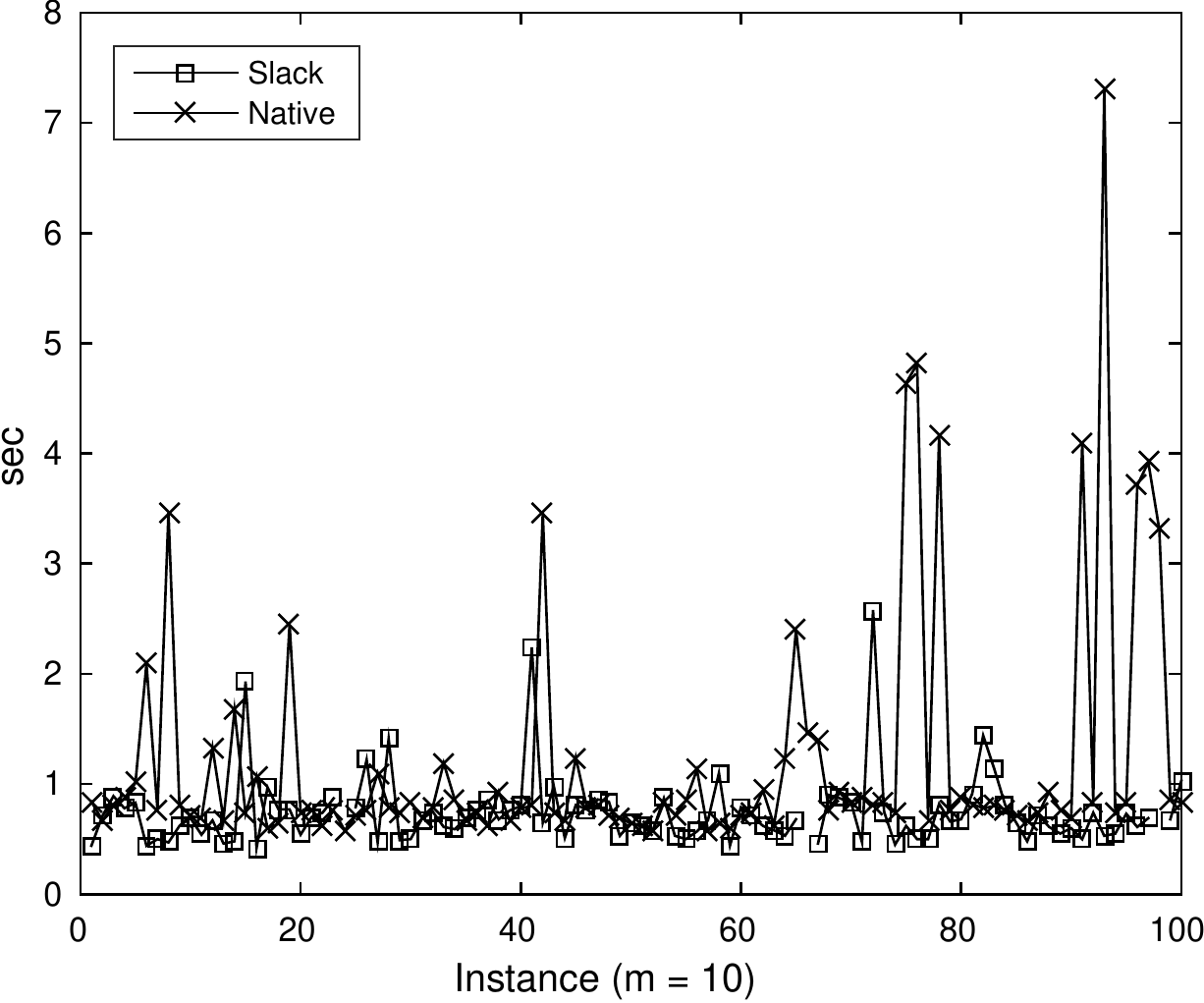}
\end{subfigure}
\begin{subfigure}{}
\includegraphics[scale=0.7]{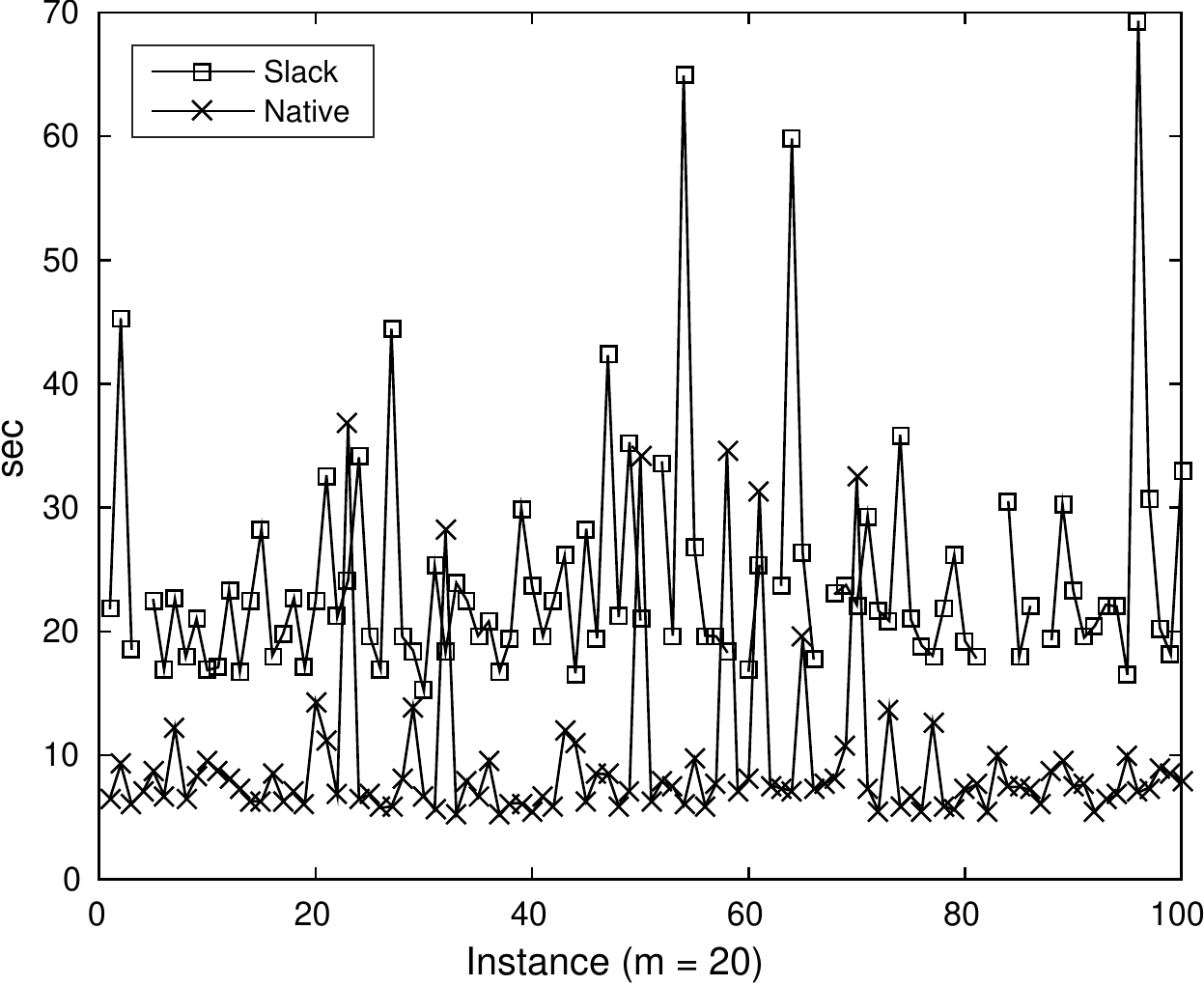}
\end{subfigure}
\caption{\ref{eq:cor_ext1} vs \ref{eq:cor_ext2}. Instance-by-instance running times for $m = 10$ and $m = 20$.
Failures are represented by omitting the corresponding running time.}\label{fig:corr_ext}
\end{figure}


\section{Final remarks}
\label{sec:conclusion}

In this article, we have shown that the optimality conditions for \eqref{eq:sdp} 
and \eqref{eq:sdps} are essentially the same. One intriguing part of this connection is 
the fact that the addition of squared slack variables seems to be enough to capture 
a great deal of the  structure of $\PSDcone{m}$. The natural progression from here is to expand 
the results to symmetric cones. In this article, we already saw some 
results that have a distinct Jordan-algebraic flavor, such as Lemma \ref{lemma:psd}.
It would be interesting to see how  these results can be further extended  and, 
whether clean proofs can be obtained without recoursing to the classification of simple Euclidean 
Jordan algebras. 

As for the computational results, we found it mildly surprising that the slack 
variables approach was able to outperform the ``native'' approach in many instances. This warrants 
a deeper investigation of whether this could be a reliable tool for attacking 
NSDPs that are not linear. These are precisely the ones that are not 
covered by the earlier work by Burer and Monteiro \cite{BM03,BM05}.


\bibliographystyle{spmpsci}
\bibliography{journal-abbrev,references}

\end{document}